\documentclass[leqno]{amsart}
\usepackage{amsmath}
\usepackage{amssymb}
\usepackage{amsthm}
\usepackage{enumerate}
\usepackage[mathscr]{eucal}
\usepackage{float}
\theoremstyle{plain}
\usepackage{tikz}
\newtheorem{theorem}{Theorem}[section]
\newtheorem{lemma}[theorem]{Lemma}

\newtheorem{corollary}{Corollary}
\theoremstyle{definition}
\newtheorem{definition}[theorem]{Definition}

\newtheorem{conjecture}[theorem]{Conjecture}
\newtheorem{example}[theorem]{Example}

\theoremstyle{remark}

\begin{document}
	
	\title [Birkhoff-James extensions of continuous functions]{Birkhoff-James extensions of continuous functions on  metric spaces}
	
	\author[Saptak Bhattacharya]{Saptak Bhattacharya}
	
	\address[Bhattacharya]{Department of Mathematics\\ Jadavpur University\\ Kolkata 700032\\ India}
	\email{saptak1284@gmail.com}

	%\thanks will become a 1st page footnote.
	\thanks{Saptak Bhattacharya would like to thank the National Board of Higher Mathematics for supporting him financially during his Masters}

	%    General info

	%    General info

	\date{}
	\maketitle
    \begin{abstract} 
    	 In this paper, we extend the investigations regarding Birkhoff-James orthogonality of linear operators to bounded continuous functions on metric spaces. We introduce Birkhoff-James extensions of continuous functions and study them in detail, in the separate contexts of compact and non-compact metric spaces. We conclude by discussing an application of our ideas to the study of Birkhoff-James orthogonality in $C(X)$ with the supremum norm, where $X$ is a compact metric space.
    \end{abstract}
	
	\section{Introduction}
	%\vspace{5mm}
	
	Given two elements $x$ and $y$ of a nomed linear space $\mathbb{X}$ over $\mathbb{F}$($\mathbb{R}$ or $\mathbb{C}$), we say $y$ is Birkhoff-James orthogonal to $x$, or $x\perp_{B}y$ , if $||x + \lambda y||\geq||x||$ for all $\lambda \in\mathbb{F}$. This generalizes the concept of orthogonality in inner product spaces to arbitrary normed spaces. First introduced by Birkhoff in \cite{GB}, this notion has played an important role in the study of the geometry of normed linear spaces. James \cite{RCJ, RCJa} used it extensively in his study of geometric properties of normed linear spaces like smoothness and strict convexity. Recent studies \cite{MPRS, RBPS, SP, SPMR, SPM} of this phenomenon in the context of operator theory have revealed a lot about the geometry of linear operators. Indeed, if $T$ and $U$ be a pair of bounded linear operators on a Banach space $\mathbb{B}$ over $\mathbb{F}$($\mathbb{R}$ or $\mathbb{C}$), we say $T\perp_{B}U$ if $||T + \lambda U||\geq||T||$ for all $\lambda\in\mathbb{F}$, where $||.||$ denotes the usual operator norm.\medskip
	%\vspace{3mm}
	 
	Bhatia and \v{S}emrl \cite{RBPS} and Paul \cite{KP} independently showed that given two linear operators $T$ and $U$ on a finite dimensional inner product space $\mathbb{X}$, $T\perp_{B}U$ if and only if there exists some $x\in X$ with $||x||=1$ and $||Tx||=||T||$ such that $\langle Tx, Ux \rangle=0$. It had been conjectured by Bhatia and \v{S}emrl \cite{RBPS} that if $T$ is an operator on a finite dimensional normed linear space $\mathbb{X}$, $T\perp_{B} U$ for some operator $U$ on $\mathbb{X}$ implies the existence of some $x_{0}$ such that $||Tx_{0}||=||T||||x_{0}||$ and $Tx_{0} \perp_{B} Ux_{0}$. This didn't turn out to be true, as Li and Schnieder \cite{LS} gave a counterexample in 2002. This property of linear operators is now known as the Bhatia-\v{S}emrl property, see \cite{SPH}. Sain and Paul \cite{SP} proved that given an operator $T$ on a finite dimensional real normed space $\mathbb{X}$ with unit sphere $S_{\mathbb{X}}$, $T$ satisfies the Bhatia-\v{S}emrl property if  $M_{T}=\{x \in S_{\mathbb{X}}: ||Tx||=||T||\}$ is connected under the projective identification $x\sim-x$.
	\medskip
%	\vspace{3mm}
		
	In this paper, we generalize the concept of Birkhoff-James orthogonality of linear operators to bounded continuous functions on metric spaces by introducing the notion of Birkhoff-James extensions, and give a complete characterization of the phenomenon for compact metric spaces. This generalizes the characterization of Birkhoff-James orthogonality of operators on finite dimensional normed linear spaces demonstrated in Sain \cite{DS}. We also define what it means for a continuous function on a compact metric space to satisfy the Bhatia-\v{S}emrl property, and with a proof that is different in flavour, generalize Theorem 2.1 in Sain and Paul \cite{SP}. One of the main objectives of this paper is to view these results solely from the perspective of convex analysis. In the third section, we study Birkhoff-James extensions of bounded continuous functions on non-compact metric spaces. Finally, as an application of the results obtained, we give a complete characterization of Birkhoff-James orthogonality in $C(X)$, where $X$ is a compact metric space.\medskip
	
	% \vspace{5mm}
Throughout this paper, we rely heavily on the sub-differential calculus of convex functions. 
Before we end this section we briefly mention the symbols to be used in due course. Given a function $p$ of  two variables $x$ and $t$, we denote by $p(x,\rule{3mm}{0.15mm})$ the function $t\to p(x,t)$ when $x$ is fixed. We denote by $p_{t}(x_{0}, t_{0})$ the derivative with respect to the second co-ordinate(when it exists) at the point $t_0$, keeping the first co-ordinate fixed at $x_0$. Accordingly, right and left derivatives(when they exist) shall be denoted by $p_{t+}(x_0, t_0)$ and $p_{t-}(x_0, t_0)$ respectively.

	\section{Birkhoff-James Extensions of Continuous Functions on Compact Metric Spaces}
	
	 We begin this section by defining a convex extension of a continuous function on a compact metric space.
	 
	\begin{definition} Let $(X,d)$ be a compact metric space and $f:X\to \mathbb{R}$ be a continuous function. A continuous function $p: X\times\mathbb{R} \to \mathbb{R}$ is said to be a convex extension of $f$ if the following two properties are satisfied :
	
	(i) $p(x,0)=f(x)$ for all $x\in X$.
	
	(ii) For each $x\in X$, $p(x,\rule{3mm}{0.15mm})$ is a convex function on $\mathbb{R}$.

	\end{definition}

    Recall that a  function $\phi : \mathbb{R} \to \mathbb{R}$ is said to be convex if $ \phi((1-t)x + ty ) \leq (1-t)\phi(x) + t\phi(y) $ for all $ x, y \in \mathbb{R}$ and for all $ t \in [0,1].$ \medskip
    
	Having defined a convex extension, we now define a Birkhoff-James extension of a given continuous function on a compact metric space $X$.

\begin{definition} Let $f$ be a continuous function on a compact metric space $(X,d)$. A convex extension $p: X\times\mathbb{R}\to \mathbb{R}$ is said to be a Birkhoff-James extension if $g(t)\geq g(0)$ for all $t\in \mathbb{R}$ where $g(t)=$ sup$_{x\in X}p(x,t)$.
	\end{definition}

    We note that $g$ is a convex function on $\mathbb{R}$.
	
	We observe that given any continuous function $f$ on a compact metric space $(X,d)$ and a convex function $h:\mathbb{R}\to \mathbb{R}$ such that $h(t)\geq h(0)$ for all $t\in \mathbb{R}$, the map $p:X\times\mathbb{R}\to\mathbb{R}$ defined by $p(x,t)=f(x)+h(t)$ is a Birkhoff-James extension of $f$. Infact, since the cardinality of all such functions $h$ is $2^{\aleph_{0}}$, we conclude that the cardinality of the set of all Birkhoff-James extensions of $f$ is atleast $2^{\aleph_{0}}$. Since $X$ is compact, $X \times \mathbb{R}$ is separable, and hence, the cardinality is exactly $2^{\aleph_{0}}$.
	
	The notion of Birkhoff-James extensions generalizes that of Birkhoff-James orthogonality of linear operators in the following sense : Let $S_{\mathbb{X}}$ denote the unit sphere of a finite-dimensional real normed linear space $\mathbb{X}$. Given a pair of linear operators $A$ and $B$ on $\mathbb{X}$, we consider the function $p: S_{\mathbb{X}}\times\mathbb{R}\to\mathbb{R}$ given by $p(x,t)=||Ax + tBx||$. Clearly, $p$ is a convex extension of the continuous map $x \to ||Ax||$ on $S_{\mathbb{X}}$. Then, $p$ is a Birkhoff-James extension of this map if and only if $A\perp_{B} B$.
	
	A sufficient condition for a convex extension to be a Birkhoff-James extension can be stated and proved quite simply. Let $M_{f}$ be the supremum attaining set of $f$. Then $p$ is a Birkhoff-James extension of $f$ if there exists some $\bar{x}\in M_{f}$ such that $p(\bar{x},t)\geq p(\bar{x},0)$ for all $t\in \mathbb{R}$. This can be seen from noticing that in such a case, $g(t)=$ sup$_{x\in X}p(x,t)\geq p(\bar{x},t)\geq p(\bar{x},0)=g(0)$ for all $t\in \mathbb{R}$. This condition isn't necessary, as the following example demonstrates. 
	
	\begin{example}We consider $X=[-1,1]$ with the usual metric topology, $f:[-1,1]\to \mathbb{R}$ given by $f(x)=|x|$ for all $x\in [-1,1]$ and a Birkhoff-James extension $p$ defined by $p(x,t)=|t+x|$ for all $(x,t)\in [-1,1]\times\mathbb{R}$. It's easily observed that there exists no $\bar{x}\in \{-1,1\}$ such that $p(\bar{x},t)\geq p(\bar{x},0) = g(0)$.\end{example}
	
	Now, we define what it means for a continuous function on a compact metric space to satisfy the Bhatia-\v{S}emrl property.
	
	\begin{definition} Let $X$ be a compact metric space and $f:X\to \mathbb{R}$ be continuous. Let $M_{f}=\{x\in X:f(x)=$ sup$_{y\in X}f(y)\}$. Then, $f$ is said to satisfy the Bhatia-\v{S}emrl property if for every Birkhoff-James extension $p$ of $f$, there exists $\bar{x}\in M_{f}$ such that $p(\bar{x},t)\geq p(\bar{x}, 0)=f(\bar{x})$ for all $t \in \mathbb{R}$.\end{definition}
	
	We are now going to prove a sufficient condition for a continuous function on a compact metric space to satisfy the Bhatia-\v{S}emrl property, thereby generalizing Theorem 2.1 in Sain and Paul \cite{SP}. But before that, we state an important lemma that is going to be useful for our purposes.
	
	\begin{lemma}\label{l-conv} Let $I$ be a non-empty open subinterval of $\mathbb{R}$. Let $f_{n}: I\to \mathbb{R}$ be a sequence of convex functions on $I$ converging pointwise to a function $f$. Then, $f$ is convex and for any $t\in I$, limsup$_{n\to\infty} f_{n}^{\prime}(t+)\leq f^{\prime}(t+)$.\end{lemma}
	
	\begin{proof} Given $x,y\in\mathbb{R}$, by convexity of $f_{n}$, we have $f_{n}(\lambda x+(1- \lambda)y)\leq\lambda f_{n}(x)+ (1- \lambda)f_{n}(y)$ for all $\lambda\in [0,1]$. Taking $n \to \infty$, we see that $f(\lambda x + (1- \lambda)y)\leq \lambda f(x) + (1- \lambda)f(y)$ for all $\lambda \in [0,1]$. Thus, $f$ is convex.
	
	Fixing $t\in I$, we choose $\varepsilon\textgreater0$. We now choose $\delta\textgreater0$ such that $t+\delta\in I$ and $\frac{f(t+ \delta) - f(t)}{\delta}\leq f^{\prime}(t+) + \varepsilon$. 
	
	Clearly, 
	
	$$\frac{f_{n}(t+ \delta) - f_{n}(t)}{\delta}\to \frac{f(t+ \delta) - f(t)}{\delta}\textrm{as}n\to \infty$$
	
	Hence, there exists $N\in \mathbb{N}$ such that 
	
	$$\frac{f_{n}(t+ \delta) - f_{n}(t)}{\delta}\leq\frac{f(t+ \delta) - f(t)}{\delta}+\varepsilon\leq f^{\prime}(t+) + 2\varepsilon$$ 
	
	for all $n\geq N$. 
	
	But since each $f_{n}$ is convex, $f_{n}^{\prime}(t+)\leq\frac{f_{n}(t + \delta) - f_{n}(t)}{\delta}$ for all $n\in \mathbb{N}$. Therefore, for all $n\geq N$, $f_{n}^{\prime}(t+)\leq f^{\prime}(t+) + 2\varepsilon$. Since $\varepsilon$ was arbitrary, it follows that limsup$_{n\to \infty}f_{n}^{\prime}(t+)\leq f^{\prime}(t+)$. Hence, we have proved our lemma. \end{proof}

	We are now ready to state and prove our first theorem.
	
	\begin{theorem}\label{th-suff}
		 Let $(X,d)$ be a compact metric space, $f:X\to \mathbb{R}$ be continuous and $M_{f}$ be defined as above. Then, $f$ satisfies the Bhatia-\v{S}emrl property if $M_{f}$ is connected.
	\end{theorem}
	\begin{proof} Given a Birkhoff-James extension $p$ of $f$, let $E_{m}=\{x\in M_{f}:p(x,t)\geq p(x,0)\forall t\in [-m,m]\}$ for all $m\in \mathbb{N}$. Its easy to observe that $E_{m+1}\subset E_{m}$ for all $m\in \mathbb{N}$ and each $E_{m}$ is a closed subset of $M_{f}$ and therefore, compact. Thus, if we can show that each $E_{m}$ is non-empty, it will follow that their intersection is also non-empty, thereby implying that $f$ satisfies the Bhatia-\v{S}emrl property. We shall, therefore, show that $E_{m}$ is non-empty for all $m\in \mathbb{N}$.\medskip
	
	Fixing $m\in \mathbb{N}$, we note a few basic facts about convex functions first. A convex function $h:[-m,m]\to \mathbb{R}$ attains its minima at $0$ if and only if $h^{\prime}(0+)\geq 0$ and $h^{\prime}(0-)\leq 0$, and in particular, if either $h^{\prime}(0+)$ or $h^{\prime}(0-)$ equals zero. Also, the set of points where $h$ is not differentiable is countable, and therefore, the set of differentiability points is dense in the domain interval.\medskip
	
	Since $p$ is a Birkhoff-James extension, $g(t)=$ sup$_{x\in X}p(x,t)\geq g(0)$ for all $t\in [-m,m]$. Clearly, $g$ is convex, so $g^{\prime}(0+)\geq 0$ and $g^{\prime}(0-)\leq 0$. We choose a sequence $t_{n}\textgreater 0$ such that $t_{n}\to 0+$ and $g^{\prime}(t_{n})$ exists. For each $t_{n}$, we choose $x_{n}$ such that $p(x_{n},t_{n})=$ sup$_{x\in X}p(x,t_{n})=g(t_{n})$. Then its easy to observe that $p_{t+}(x_{n}, t_{n}) \leq g^{\prime}(t_{n}+)$ and $p_{t-}(x_{n}, t_{n}) \geq g^{\prime}(t_{n}-)$. Since $g$ is differentiable at each $t_{n}$ it follows that $g^{\prime}(t_{n})=p_{t}(x_{n}, t_{n})\geq 0$ for all $n\in \mathbb{N}$. Since $X$ is compact, we can replace $x_{n}$ with a convergent subsequence of it. Due to notational convenience, we, therefore, assume that $x_{n}$ converges to some $x\in X$. Then, $p(x_{n}, t_{n})\to p(x,0)$. Again, $p(x_{n}, t_{n})=g(t_{n})$ for all $n\in \mathbb{N}$, and therefore, since $g$ is continuous, taking $n\to \infty$ we observe that $p(x,0)=g(0)$, implying that $x\in M_{f}$.\medskip
	
	Now since $p$ is continuous on $X\times[-m,m]$ with the standard product metric, it is uniformly  continuous and therefore, given any $\varepsilon\textgreater0$ there exists $\delta\textgreater0$ such that whenever $d(x,y)\leq\delta$, where $d$ is the metric on $X$, $|p(x,t)-p(y,t)|\leq \varepsilon$ for all $t\in [-m,m]$. From here, it follows that since $x_{n}\to x$, $p(x_{n}, \rule{3mm}{0.15mm})$ converges to $p(x,\rule{3mm}{0.15mm})$ uniformly on $[-m,m]$. continuous and therefore, given any $\varepsilon\textgreater 0$ there exists $\delta \textgreater 0$ such that whenever This implies, by lemma \ref{l-conv} that given any $\lambda\in (-m,m)$, limsup$_{n\to \infty}p_{t+}(x_{n}, \lambda)\leq p_{t+}(x, \lambda)$.\medskip
	
	Let $\varepsilon \textgreater 0$. Choosing $s$ such that $p_{t+}(x, s) \leq p_{t+}(x,0)+\varepsilon$, we choose $N\in \mathbb{N}$ such that $t_{n}\textless s$ for all $n\geq N$. Thus, for all $n\geq N$, $p_{t}(x_{n}, t_{n}) \leq p_{t+}(x_{n}, s) \leq p_{t+}(x, 0)+ \varepsilon$. This implies that limsup$_{n\to \infty}p_{t}(x_{n}, t_{n}) \leq p_{t+}(x, 0)$ since $\varepsilon$ was arbitrary. Now, $p_{t}(x_{n}, t_{n})\geq 0$ for all $n\in \mathbb{N}$, and therefore, $p_{t+}(x,0)\geq 0$. If $p_{t+}(x,0)=0$ then we are done, so, we assume $p_{t+}(x,0)\textgreater 0$. Let $U=\{x\in M_{f}:p_{t+}(x, 0)\textgreater 0\}$ and $V=\{x\in M_{f}:p_{t-}(x, 0)\textless 0\}$.\medskip
    
    We have already shown that $U$ is non-empty. Similarly, it follows that $V$ is non-empty too. Assuming there doesn't exist any $x\in M_{f}$ such that either $p_{t+}(x,0)$ or $p_{t-}(x,0)$ equals zero, since that would prove our theorem, and using the fact that for any $x\in X$, $p_{t-}(x,0)\leq p_{t+}(x,0)$ since $p(x, \rule{3mm}{0.15mm})$ is convex, we conclude that $M_{f}=U \cup V$. If $U$ and $V$ aren't disjoint, we are done, so we assume that they are disjoint. Now, if $x\in U$, then $p(x,0)\textgreater p(x, b)$ for some $b\in [-m, 0)$. We choose $\varepsilon\textgreater 0$ such that $p(x,b)+\varepsilon\textless p(x,0)-\varepsilon$. We choose $\delta\textgreater 0$ such that for all $y$ satisfying $d(x,y)\textless \delta$, $|p(x,t)-p(y,t)|\textless \frac{\varepsilon}{2}$ for all $t\in[-m,m]$, thereby implying that $p(y,b)\textless p(y,0)$ for all such $y$ and $p_{t+}(y,0)\textgreater 0$. Thus, $U$ is open. Similarly, $V$ is also open and $U$ and $V$ form a separation of $M_{f}$. This is a contradiction to the fact that $M_{f}$ is connected. Therefore, $f$ satisfies the Bhatia-\v{S}emrl property, and the proof is complete. \end{proof}
    
    As a simple application of Theorem \ref{th-suff}, it can be shown that the set of all continuous functions $f$ on a compact metric space $X$ satisfying the Bhatia-\v{S}emrl property is dense in $C(X)$. We state it below as a corollary.
    
    \begin{corollary} Let $(X,d)$ be a compact metric space. Then the set of all $f\in C(X)$ such that $f$ satisfies the Bhatia-\v{S}emrl property is dense in $C(X)$, which is endowed with the usual supremum norm.\end{corollary}
    
    \begin{proof}
    If $X$ is a singleton, then it's obvious. Otherwise, if $f\in C(X)$, we choose $x_{0}\in M_{f}$ and $y_{0}\in X \setminus \{x_{0}\}$. Then, for every $\varepsilon\textgreater0$ we consider the map $f_{\varepsilon}: X \to \mathbb{R}$ given by 
    
    $$f_{\varepsilon}(x)=f(x)-\frac{\varepsilon d(x, x_{0})}{d(x, x_{0}) + d(x, y_{0})}$$

    for all $x\in X$. Then $f_{\varepsilon}$ is continuous and $|f(x) - f_{\varepsilon}(x)|\leq\varepsilon$ for all $x\in X$. Also, $f_{\varepsilon}(x)\leq f(x)$, sup$_{x\in X}f_{\varepsilon}(x)=$ sup$_{x\in X}f(x)$ and $M_{f_{\varepsilon}}=\{x_{0}\}$, which is connected. Thus, by Theorem \ref{th-suff}, $f_{\varepsilon}$ satisfies the Bhatia-\v{S}emrl property. Since $\varepsilon\textgreater 0$ was arbitrarily chosen, we are done.\end{proof}
    
    Our approach to Theorem \ref{th-suff} also helps us give a necessary and sufficient condition for a convex extension of a continuous function on a compact metric space to be a Birkhoff-James extension in terms of subdifferentials of the functions $p(x,t)$. We state it in the next theorem below.
    
    \begin{theorem}\label{th-char} Let $(X,d)$ be a compact metric space, $f:X\to\mathbb{R}$ be continuous, and $p:X\times \mathbb{R}\to \mathbb{R}$ be a convex extension of $f$. Then, $p$ is a Birkhoff-James extension if and only if their exist $x,y\in M_{f}$ such that $p_{t+}(x, 0)\geq0$ and $p_{t-}(y, 0)\leq0$.\end{theorem}
    
    \begin{proof} First, we show that the condition is necessary. As before, we choose a suitable sequence $x_{n}$ converging to some $x\in M_{f}$ and a sequence of positive reals $t_{n}\to 0+$ such that $g(t)=$ sup$_{x\in X}p(x,t)$ is differentiable at each $t_{n}$. Just like the proof of Theorem \ref{th-suff}, we have $p_{t}(x_{n}, t_{n})=g^{\prime}(t_{n})$ for all $n$. Therefore, limsup$_{n\to \infty}p_{t}(x_{n}, t_{n})=g^{\prime}(0+)\leq p_{t+}(x,0)$. But $g$ is convex and attains its minima at $0$. Thus, $p_{t+}(x, 0)\geq g^{\prime}(0+)\geq0$. Similarly, we can assert the existence of $y\in M_{f}$ such that $p_{t-}(y, 0)\leq0$, thereby showing what was required.\medskip
    
    To show that the condition is sufficient, we simply observe that whenever such $x$ and $y$ are given, $g^{\prime}(0+)\geq p_{t+}(x,0)\geq0$ and $g^{\prime}(0-)\leq p_{t-}(x,0)\leq0$, thus implying that $g(t)\geq g(0)$ for all reals $t$ and completing the proof. \end{proof}
    
    Theorem \ref{th-char}, a generalization of Theorem 2.2 in Sain \cite{DS}, completely characterizes Birkhoff-James extensions of continuous functions on compact metric spaces and basically states that whenever a Birkhoff-James extension $p$ is given, we separately have $x$ and $y$ in $M_{f}$ such that $p_{t+}(x, 0)\geq0$ and $p_{t-}(y, 0)\leq0$. When there exists one $x$ in $M_{f}$ such that both $p_{t+}(x, 0)\geq0$ and $p_{t-}(x, 0)\leq0$ the Bhatia-\v{S}emrl property is satisfied.\medskip
    
    In the following section, we study Birkhoff-James extensions of bounded continuous functions on non-compact metric spaces. There, after giving the definitions, we give an example to show that such a function defined on a non-compact metric space might not satisfy the Bhatia-\v{S}emrl property even if its supremum-attaining set is non-empty and connected, thereby demonstrating the necessity of all the conditions in Theorem \ref{th-suff}. Subsequently, we give a necessary condition for a convex extension of a bounded continuous function on a non-compact metric space to be a Birkhoff-James extension. Finally, we conclude the section with a sufficient condition regarding the same, demonstrated as an application of Lemma \ref{l-conv}.

    \section{Birkhoff-James Extensions of Continuous Functions on Non-Compact Metric Spaces}

    Our definitions of convex and Birkhoff-James extensions are going to be a bit different when the domain is non-compact. We shall be introducing a useful property in order to compensate a bit for the loss of compactness. That doesn't limit the scope of our work in any way, since we already have that property built in the study of Birkhoff-James orthogonality of bounded linear operators on Banach spaces.
    
    \begin{definition}\label{def-nc} Let $(X,d)$ be a non-compact metric space and $f:X\to \mathbb{R}$ be a bounded continuous function. A continuous map $p: X\times \mathbb{R}\to \mathbb{R}$ is said to be a convex extension of $f$ if the following properties hold :
    
    (i) For each $t\in\mathbb{R}$, $p(\rule{3mm}{0.15mm}, t)$ is a bounded function on $X$.
    
    (ii) For each $x\in X$, $p(x,0)=f(x)$ and $p(x, \rule{3mm}{0.15mm})$ is a convex function on $\mathbb{R}$.
    
    (iii) There exists a continuous map $h: \mathbb{R}\to \mathbb{R}$ such that $h(0)$ = $0$ and $|p(x,t) - p(x,s)|\leq|h(t) - h(s)|$ for all $x\in X$ and $t,s\in\mathbb{R}$.\end{definition}

    If $\mathbb{X}$ is an infinite dimensional normed linear space, it's unit sphere $S_{\mathbb{X}}$ is not compact. Given two bounded linear operators $A$ and $B$ on $\mathbb{X}$, we consider the functions $f: S_{\mathbb{X}}\to \mathbb{R}$ and $p:S_{\mathbb{X}}\times \mathbb{R}\to\mathbb{R}$ given by $f(x)=||Ax||$ for all $x\in S_{\mathbb{X}}$ and $p(x,t)=||Ax + tBx||$ for all $(x,t)\in S_{\mathbb{X}}\times\mathbb{R}$. Then, $p$ is a convex extension of $f$. Properties (i) and (ii) are obvious. Property (iii) is satisfied because for all $x\in S_{\mathbb{X}}$, $t,s\in\mathbb{R}$, $|||Ax + tBx||- ||Ax + sBx|||\leq|t-s|||Bx||\leq|t-s|||B||$ by triangle inequality and boundedness of $B$.

    \begin{definition} Let $f$ be a bounded continuous function on a non-compact metric space $(X,d)$. A convex extension $p: X\times \mathbb{R}\to \mathbb{R}$ is said to be a Birkhoff-James extension if $g(t)\geq g(0)$ for all $t\in\mathbb{R}$ where $g(t)=$ sup$_x\in Xp(x,t)$ for all $t\in\mathbb{R}$.\end{definition}

    Clearly, this generalizes the notion of Birkhoff-James orthogonality of bounded linear operators on infinite dimensional normed spaces. \medskip

    The definition of the Bhatia-\v{S}emrl property remains unchanged and gets carried over from the previous section. As mentioned before, we now give an example to demonstrate that a bounded continuous function $f$ on a non-compact metric space might not satisfy the Bhatia-\v{S}emrl property even if the supremum attaining set of $f$ is non-empty and connected.

     \begin{example}Let $X=[0, \infty)$. We consider the constant function $f \equiv1$ on $X$. The map $p:X\times \mathbb{R}\to \mathbb{R}$ given by $p(x,t)=1+te^{-x}$ is a convex extension of $f$. Also, since $e^{-x}\leq e^{0}=1$ and $e^{-x}\to 0$ as $x\to \infty$, $g(t)=$ sup$_{x\in X}p(x,t)=\textrm{max}\{1, 1+t\}$ for all $t\in \mathbb{R}$. Therefore, $p$ is a Birkhoff-James extension of $f$, but for any $x\in X$, $p(x, t)=1+te^{-x}$ for all $t\in \mathbb{R}$, which is unbounded below as a function of $t$. Therefore, the Bhatia-\v{S}emrl property is not satisfied even though $M_{f}=[0, \infty)$, which is non-empty and connected.\end{example}

     In our next theorem, we give another necessary condition for a given convex extension of a bounded continuous function $f$ to be a Birkhoff-James extension, this time assuming that the domain is non-compact, which means, in particular, that $M_{f}$ can be empty. But first, we define a maximizing sequence.
    
    \begin{definition} Let $(X,d)$ be a non-compact metric space and $f:X \to \mathbb{R}$ be a bounded continuous function. A sequence $\{x_{n}\}$ is said to be a maximizing sequence for $f$ if $f(x_{n})\to $ sup$_{x\in X}f(x)$.\end{definition}
    
    Having stated the definition, we now state and prove our next theorem.
    
    \begin{theorem}\label{th-ness} Let $(X,d)$ be a non-compact metric space and $f:X\to \mathbb{R}$ be bounded and continuous. Let $p:X\times \mathbb{R}\to \mathbb{R}$ be a Birkhoff-James extension of $f$. Let $g(t)=$ sup$_{x\in X}p(x,t)$ for all $x\in \mathbb{R}$. Then :-
    
    (i) There exists a maximizing sequence $\{x_{n}\}$ of $f$ and a sequence of positive reals $t_{n}\to 0+$ such that both the sequences $p_{t+}(x_{n}, t_{n})$ and $p_{t-}(x_{n}, t_{n})$ converge to $g^{\prime}(0+)\geq0$.
   
    (ii) There exists a maximizing sequence $\{y_{n}\}$ of $f$ and a sequence of negative reals $s_{n}\to 0+$ such that both the sequences $p_{t+}(y_{n}, s_{n})$ and $p_{t-}(y_{n}, s_{n})$ converge to $g^{\prime}(0-)\leq0$.\end{theorem}
    
    \begin{proof} We only prove (i), as (ii) follows similarly. To begin, we choose a sequence $t_{n}$ of positive reals such that $g^{\prime}(t_{n})$ exists. For each $n$ we choose a sequence $\{x_{n_{k}}\}$ indexed by $k$ such that $p(x_{n_{k}}, t_{n})\uparrow g(t_{n})$. We now choose $h\textgreater0$ sufficiently small such that $\frac{g(t_{n}+h)-g(t_{n})}{h}\leq g^{\prime}(t_{n})+ \frac{1}{2^{n+1}}$ and $\frac{g(t_{n})-g(t_{n}- h)}{n}\geq g^{\prime}(t_{n}) - \frac{1}{2^{n+1}}$.\medskip
    
    Now, let $x_{n}$ be an element of the sequence $\{x_{n_{k}}\}$ such that $g(t_{n})\leq p(x_{n}, t_{n})+ \frac{h}{2^{n+1}}$. Then,
    
    $$\frac{g(t_{n}+h)-g(t_{n})}{h}\geq\frac{p(x_{n}, t_{n}+h) - p(x_{n}, t_{n})}{h}-\frac{1}{2^{n+1}}$$
    
    We also have 
    
    $$\frac{p(x_{n}, t_{n}) - p(x_{n}, t_{n}-h)}{h}\geq \frac{g(t_{n})- g(t_{n}-h)}{h}-\frac{1}{2^{n+1}}$$
    
    Together, they imply that $p_{t+}(x_{n}, t_{n})\leq g^{\prime}(t_{n}) + \frac{1}{2^{n}}$ and $p_{t-}(x_{n}, t_{n})\geq g^{\prime}(t_{n}) - \frac{1}{2^{n}}$. But $g^{\prime}(t_{n})\to g^{\prime}(0+)$. Therefore, we conclude that limsup$_{n\to \infty}p_{t+}(x_{n}, t_{n})\leq g^{\prime}(0+)$ and liminf$_{n\to\infty}p_{t-}(x_{n}, t_{n})\geq g^{\prime}(0+)$.

    The fact that $p_{t-}(x_{n}, t_{n})\leq p_{t+}(x_{n}, t_{n})$ for all $n$ now implies that both the sequences $p_{t+}(x_{n}, t_{n})$ and $p_{t-}(x_{n}, t_{n})$ converge to $g^{\prime}(0+)$.
    
    It only remains to be shown that the sequence $\{x_{n}\}$ is maximizing. To prove this, we note that $g(t_{n}) - \frac{1}{2^{n+1}}\leq p(x_{n}, t_{n})\leq g(t_{n})$. Thus, by continuity of $g$, $p(x_{n}, t_{n})\to g(0)$. Now by property (iii) of Definition \ref{def-nc}, $|p(x_{n}, t_{n}) - p(x_{n}, 0)|\to 0$. Thus, $p(x_{n}, 0)\to g(0)$, thereby showing that the sequence $\{x_{n}\}$ is maximizing, and completing the proof.\end{proof}
    
    The necessary condition thus obtained can be seen as the non-compact counterpart of the necessary part of Theorem \ref{th-char}. Now, as another application of Lemma \ref{l-conv}, we give a sufficient condition for a convex extension of a bounded continuous function on a non-compact metric space to be a Birkhoff-James extension.
    
    \begin{theorem}\label{nc-suff} Let $(X,d)$ be a non-compact metric space and $f:X\to \mathbb{R}$ be bounded and continuous. Let $p:X\times\mathbb{R}\to \mathbb{R}$ be a convex extension of $f$. Then, $p$ is a Birkhoff-James extension if there exist maximizing sequences $\{x_{n}\}$ and $\{y_{n}\}$ of $f$, a $\delta\textgreater 0$, and two sequences $t_{n}\to 0+$ and $s_{n}\to 0-$ such that $p(x_{n}, \rule{3mm}{0.1mm})$ and $p(y_{n}, \rule{3mm}{0.1mm})$ converge pointwise to $g$ on $(-\delta, \delta)$, limsup$_{n\to \infty}p_{t+}(x_{n}, t_{n})\geq0$ and liminf$_{n\to \infty} p_{t-}(y_{n}, s_{n}-)\leq0$.\end{theorem}
    
    \begin{proof} First, we show that $g^{\prime}(0+)\geq0$. Let $\varepsilon\textgreater0$. We choose $t_{0}\in (0, \delta)$ such that $g^{\prime}(t_{0}+)\leq g^{\prime}(0+)+ \varepsilon$. Since $p(x_{n}, t)$ converges to $g(t)$ for all $t\in (-\delta, \delta)$, Lemma \ref{l-conv} implies that limsup$_{n\to\infty} p_{t+}(x_{n}, t_{0})
    \leq g^{\prime}(t_{0}+)$. Choosing $N\in \mathbb{N}$ such that $t_{n}\textless t_{0}$ for all $n\geq N$, we see that $p_{t+}(x_{n}, t_{n})\leq p_{t+}(x_{n}, t_{0})$ for all $n\geq N$. Therefore, limsup$_{n\to \infty}p_{t+}(x_{n}, t_{n})\leq$ limsup$_{n\to \infty}p_{t+}(x_{n}, t_{0})\leq g^{\prime}(t_{0}+)\leq g^{\prime}(0+)+ \varepsilon$. Since limsup$_{n\to\infty}p_{t+}(x_{n}, t_{n})\geq0$, it follows that $g^{\prime}(0+) + \varepsilon \geq0$. Since $\varepsilon\textgreater0$ was arbitrary, $g^{\prime}(0+)\geq0$.
    
    Similarly, using the sequence $\{y_{n}\}$ in place of $\{x_{n}\}$ and the sequence $\{s_{n}\}$ in place of $t_{n}$, we conclude that $g^{\prime}(0-)\leq0$, thus showing that $p$ is a Birkhoff-James extension and completing the proof. \end{proof}
    
    We conclude this section with the following conjecture :
    
    \begin{conjecture} The conclusion of Theorem \ref{nc-suff} fails to hold if the assumption regarding the pointwise convergence of $p(x_{n}, \rule{3mm}{0.15mm})$ and $p(y_{n}, \rule{3mm}{0.15mm})$ to $g$ on $(-\delta, \delta)$ for some $\delta\textgreater0$ is removed.\end{conjecture}
    
    \section{An Application}
    
    As an application of some of the ideas presented in this paper, we now offer a simple characterization of Birkhoff-James orthogonality in $C(X)$, the space of all continuous real valued functions on $X$ with the supremum norm, where $(X,d)$ is a compact metric space.\medskip
    
   We state the result below and prove it, thereby concluding the paper.\medskip

   We denote by sgn the signum function on $\mathbb{R}$.
    
    \begin{theorem}\label{th-orth_char} Let $f,g \in C(X)$ where $(X,d)$ is a compact metric space, such that $f$ is not identically $0$. Then, $f\perp_{B}g$ if and only if there exist $x,y \in M_{|f|}=\{x\in X:|f(x)|=||f||\}$ such that $f(x)g(x)\geq 0$ and $f(y)g(y)\leq 0$.
    \end{theorem}
    
    \begin{proof} Given continuous functions $f$ and $g$ on $X$ such that $f$ is not identically zero, let $p: X\times \mathbb{R}\to \mathbb{R}$ be defined by $p(x,t)=|f(x)+ tg(x)|$ for all $(x,t)\in X\times \mathbb{R}$. Then it's easy to observe that $f\perp_{B} g$ if and only if $p$ is a Birkhoff-James extension of $|f|$. By Theorem \ref{th-char}, $f\perp_{B}g$ if and only if there exist $x,y\in M_{|f|}$ such that $p_{t+}(x,0)\geq 0$ and $p_{t-}(y,0)\leq 0$. 
    
    Since both $x$ and $y$ are from the norm attaining set of $f$, neither $f(x)$ nor $f(y)$ is $0$. Now, $p_{t+}(x,0)\geq 0$ if and only if $$\lim_{t\to 0+}\frac{|f(x) + tg(x)| - |f(x)|}{t}=g(x)\textrm{sgn}(f(x))\geq 0$$
    
    But $g(x)$sgn$(f(x))\geq 0$ if and only if $g(x)f(x)\geq 0$. Hence, $p_{t+}(x,0)\geq 0$ if and only if $g(x)f(x)\geq 0$. Similarly, $p_{t-}(y,0)\leq 0$ if and only if $f(y)g(y)\leq 0$. Thus, $f\perp_{B}g$ if and only if there exist $x,y\in M_{|f|}$ such that $f(x)g(x)\geq 0$ and $f(y)g(y)\leq 0$. This completes the proof.\end{proof}

    \noindent  \textbf{Acknowledgement}.\\
    The author would like to express his sincere gratitude to Prof. Kallol Paul, Department  of Mathematics, Jadavpur University, Kolkata, India,  who went through the manuscript meticuluously and suggested a lot of necessary improvements regarding the presentation of the material. He has been immensely helpful throughout this article's course of preparation.

    \bibliographystyle{amsplain}

\end{document}